\documentclass[11pt]{amsart}

\usepackage[T1]{fontenc}
\usepackage{lmodern}
\usepackage{microtype}
\usepackage{mathtools}
\usepackage{amssymb}
\usepackage{enumitem}
\usepackage[hidelinks]{hyperref}
\usepackage[nameinlink,noabbrev]{cleveref}

\AddToHook{env/thebibliography/before}{\clearpage}

\allowdisplaybreaks
\makeatletter
\g@addto@macro\normalsize{%
  \setlength{\abovedisplayskip}{4pt plus 0.67pt minus 0.67pt}%
  \setlength{\belowdisplayskip}{4pt plus 0.67pt minus 0.67pt}%
  \setlength{\abovedisplayshortskip}{1.33pt plus 0.67pt minus 0.33pt}%
  \setlength{\belowdisplayshortskip}{2.67pt plus 0.67pt minus 0.67pt}%
}
\makeatother
\setlist[enumerate]{label=(\roman*),leftmargin=2.1em,
  topsep=2pt plus 0.67pt minus 0.67pt,itemsep=0.67pt,
  parsep=0pt,partopsep=0pt}

\newcommand{\N}{\mathbb N}
\newcommand{\Z}{\mathbb Z}
\newcommand{\Q}{\mathbb Q}
\newcommand{\cP}{\mathcal P}
\newcommand{\cI}{\mathcal I}
\newcommand{\fd}{\operatorname{fd}}
\newcommand{\radQ}{\operatorname{rad}_{\Q}}
\newcommand{\Gal}{\operatorname{Gal}}
\newcommand{\lcm}{\operatorname{lcm}}
\newcommand{\Res}{\operatorname{Res}}

\newtheorem{theorem}{Theorem}[section]
\newtheorem{proposition}[theorem]{Proposition}
\newtheorem{lemma}[theorem]{Lemma}
\newtheorem{corollary}[theorem]{Corollary}
\newtheorem{question}[theorem]{Question}
\newtheorem*{hrsaturation}{Halberstam--Richert saturation theorem}
\newtheorem*{stronghensel}{Strong Hensel lemma}

\theoremstyle{definition}
\newtheorem{definition}[theorem]{Definition}
\newtheorem{example}[theorem]{Example}
\newtheorem{remark}[theorem]{Remark}

\crefname{theorem}{Theorem}{Theorems}
\crefname{proposition}{Proposition}{Propositions}
\crefname{lemma}{Lemma}{Lemmas}
\crefname{corollary}{Corollary}{Corollaries}
\crefname{definition}{Definition}{Definitions}
\crefname{example}{Example}{Examples}
\crefname{remark}{Remark}{Remarks}
\crefname{question}{Question}{Questions}

\title[Differences of polynomial divisor counts]{Intersective polynomials
and differences of divisor counts}
\author{Zihan Zhang}

\address{(Zihan Zhang) School of Mathematical Sciences and LPMC, Nankai University, Tianjin
300071, People's Republic of China}
\email{2211056@mail.nankai.edu.cn}


\begin{document}

\begin{abstract}
Let \(d_A(m)\) count the elements of \(A\subset\mathbb N\) that divide
a nonzero integer \(m\).  We ask when
\(|d_A(F(n))-d_A(G(n))|\) is unbounded as \(n\to\infty\) for every
infinite \(A\).  If the nonzero integer polynomials \(F,G\) are coprime
over \(\mathbb Q\), this holds if and only if one is intersective,
that is, has a root modulo every positive integer.  We also prove a
version that compares \(F\) with finitely many polynomials.  We follow
Ding's proof, using the Halberstam--Richert theorem and a bound on fixed
divisors.

For pairs with common factors, let \(U,V\) be the squarefree products
of the irreducible factors found only in \(F\) and only in \(G\),
respectively.  This property requires \(UV\) to have a root modulo all but finitely
many primes.  For pairs of degree at most two, the property holds if and
only if \(UV\) and at least one of \(F,G\) are intersective.
The same criterion holds for
\(F=\gamma L^{r_L}U^{r_U}\) and \(G=\delta L^{s_L}V^{s_V}\).
Here \(L,U,V\) are pairwise nonassociate primitive linear polynomials.
The constants \(\gamma,\delta\) are nonzero integers, and the four
exponents are positive integers.
\end{abstract}

\maketitle

\section{Introduction}

For an infinite set \(A\subset\N=\{1,2,\ldots\}\), S\'ark\"ozy
asked whether the numbers of elements of \(A\) dividing \(n\) and
\(n+1\) must differ by an unbounded amount \cite{Sarkozy2001}.
Ding answered this question and classified the linear pairs with this
property.  He then proved a version for finitely many linear polynomials
\cite{DingLinear,DingProblem26}.  He also asked for a classification
of polynomial pairs \cite[Problem~3]{DingLinear}.

We solve this problem for coprime pairs of any degree and for all pairs
of degree at most two.  We also treat pairs formed from three distinct
linear factors, with arbitrary positive powers.

Our proof for coprime pairs follows Ding's proof in
\cite[Theorem~4]{DingProblem26}.  We first choose an arithmetic
progression on which selected elements of \(A\) divide one polynomial
value.  We then use the Halberstam--Richert theorem to bound the number
of prime factors of the other polynomial values.  The bounds must hold
as the progression changes.  For this, we bound the fixed divisors of
the restricted polynomials.  Pairs with common factors also require
conditions on divisibility by prime powers.

A polynomial is \emph{intersective} if it has a root modulo every
positive integer.  This condition appears in
\cite[p.~335, item~(2)]{KamaeMendesFrance} and is used in polynomial
recurrence and difference-set problems
\cite{Lucier2006,BergelsonLeibmanLesigne,Le2010}.  We also use a weaker
condition: a polynomial is \emph{prime-covering} if it has a root modulo
all but finitely many primes.

For any \(A\subset\N\) and \(m\in\Z\), put
\[
  d_A(m)=\#\{a\in A:a\mid m\},
\]
In particular, \(d_A(0)=|A|\), which may be infinite.
For infinite \(A\), define
\begin{equation}\label{eq:extended-difference}
  \delta_A(u,v)=
  \begin{cases}
    0,&u=v=0,\\
    \infty,&\text{exactly one of \(u,v\) is zero},\\
    |d_A(u)-d_A(v)|,&uv\ne0.
  \end{cases}
\end{equation}
For \(F,G\in\Z[x]\), let \(\cP(F,G)\) denote the property
\begin{equation}\label{eq:property-P}
  \text{for every infinite \(A\subset\N\)},\qquad
  \limsup_{n\to\infty}\delta_A(F(n),G(n))=\infty.
\end{equation}
If \(F,G\ne0\), then \(F(n)G(n)\ne0\) for all large \(n\).
Thus the convention at zero does not affect \(\cP(F,G)\).

The first theorem treats coprime pairs.

\begin{theorem}\label{thm:coprime-classification}
Let \(F,G\in\Z[x]\setminus\{0\}\) be coprime in \(\Q[x]\).  Then
\[
  \cP(F,G)
  \quad\Longleftrightarrow\quad
  \text{\(F\) or \(G\) is intersective}.
\]
\end{theorem}

The next theorem proves the sufficient direction.

\begin{theorem}\label{thm:simultaneous-dominance}
Let \(s\ge1\) and \(F,G_1,\ldots,G_s\in\Z[x]\setminus\{0\}\).  Suppose that
\(F\) has a nonconstant intersective divisor \(H\in\Z[x]\) such that
\[
  \gcd\!\left(H,\prod_{j=1}^sG_j\right)=1
  \qquad\text{in \(\Q[x]\)}.
\]
There is a constant \(C=C(H,G_1,\ldots,G_s)\) such that:
\begin{enumerate}
  \item If \(A\subset\N\) is finite and \(|A|>C\), then for infinitely many
  positive integers \(n\),
  \[
    d_A(F(n))>\max_{1\le j\le s}d_A(G_j(n)).
  \]
  \item If \(A\subset\N\) is infinite, then
  \[
    \limsup_{n\to\infty}
    \left(d_A(F(n))-\max_{1\le j\le s}d_A(G_j(n))\right)=\infty.
  \]
\end{enumerate}
In both cases, we may choose \(n\) so that
\(\prod_jG_j(n)\ne0\).
\end{theorem}

Ding's theorem \cite[Theorem~4]{DingProblem26} treats \(F(x)=x\) and
the linear polynomials \(b_i x\pm1\).  Our theorem allows polynomial
values under the stated conditions.

\Cref{sec:sieve} gives the two bounds used in the proof.
The almost-prime bound comes from
\cite[Theorem~10.11]{HalberstamRichert}.  We apply it to integer-valued
polynomials after removing their fixed divisors.  We also show that
\(\fd(K(a+M\,\cdot))\) is bounded in terms of \(H,K\) alone when
\(K=\prod_jG_j\) and \(M\mid H(a)\).  These bounds do not depend on
the chosen elements of \(A\).

To study pairs with common factors, write \(\radQ(P)\) for the product
of the distinct primitive irreducible factors of \(P\ne0\), each with
positive leading coefficient.  An empty product is \(1\).  Let \(C\)
be the product of the factors common to \(F,G\).  Let \(U,V\) be the
products of the factors found only in \(F\) and only in \(G\),
respectively. In the products defining \(C,U,V\), each irreducible factor
is included only once.  Its multiplicity in \(F\) or \(G\)
does not affect these products. Then
\[
  \radQ(F)=CU,\qquad \radQ(G)=CV,
\]
and \(C,U,V\) are pairwise coprime. We first gives a necessary condition.

\begin{theorem}\label{thm:common-support}
Let \(F,G\in\Z[x]\setminus\{0\}\), with \(C,U,V\) as above.
If there are infinitely many primes \(p\) for which \(UV\) has no root
modulo \(p\), then \(\cP(F,G)\) fails.  Thus
\[
  \cP(F,G)\quad\Longrightarrow\quad
  UV\text{ has a root modulo all but finitely many primes}.
\]
\end{theorem}

Suppose \(UV\) is nonconstant.  Let \(L\) be its splitting field and
\(\Gamma=\Gal(L/\Q)\).  For all but finitely many primes \(p\),
\(UV\) has a root modulo \(p\) if and only if a Frobenius element
fixes a root.  Chebotarev's theorem then gives the following result.

\begin{corollary}\label[corollary]{cor:derangement}
Suppose that \(UV\) is nonconstant.  If \(\Gamma\) contains an element fixing
no root of \(UV\), then
\(\cP(F,G)\) does not hold.  Thus \(\cP(F,G)\) forces \(\Gamma\) to be the
union of the point stabilizers in its action on the roots of \(UV\).
\end{corollary}

The next theorem treats pairs with one common linear factor and one
other linear factor each.  Each factor may occur to any positive power.

\begin{theorem}
\label{thm:linear-defects}
Let
\[
  L(x)=\ell x+\lambda,\qquad
  U(x)=ux+\mu,\qquad
  V(x)=vx+\nu
\]
be pairwise nonassociate primitive linear polynomials in \(\Z[x]\).  Let \(\gamma,\delta\in\Z\setminus\{0\}\), let
\(r_L,r_U,s_L,s_V\) be positive integers, and put
\[
  F=\gamma L^{r_L}U^{r_U},\qquad
  G=\delta L^{s_L}V^{s_V}.
\]
Then
\[
  \cP(F,G)\quad\Longleftrightarrow\quad
  \begin{aligned}
    \gcd(|u|,|v|)&=1,\\
    \min\{\gcd(|\ell|,|u|),\gcd(|\ell|,|v|)\}&=1.
  \end{aligned}
\]
Equivalently, \(UV\) and at least one of \(F,G\) are intersective.
\end{theorem}

The criterion does not depend on \(\gamma,\delta\) or the four exponents.
The proof splits \(A\) into two sets using prime-power divisibility.
One case uses a sieve; the other uses an exact calculation of valuations. The above result can help us determine the intersection of polynomials \(F\) and \(G\) within no more than two iterations.

\begin{corollary}
\label[corollary]{cor:degree-two}
Let \(F,G\in\Z[x]\setminus\{0\}\) satisfy
\(\deg F,\deg G\le2\), with \(C,U,V\) as above.  Then
\[
  \cP(F,G)
  \quad\Longleftrightarrow\quad
  \text{\(UV\) and at least one of \(F,G\) are intersective}.
\]
Moreover,
\[
  \cP(F,G)
  \quad\Longleftrightarrow\quad
  \cP\bigl(\radQ(F),\radQ(G)\bigr).
\]
\end{corollary}

\begin{corollary}
\label[corollary]{cor:normalized-linear}
For \(a,b\in\Z\setminus\{0\}\),
\[
  \cP\bigl(x(ax+1),x(bx+1)\bigr)
  \quad\Longleftrightarrow\quad
  \gcd(|a|,|b|)=1\ \text{ and }\ a\ne b.
\]
\end{corollary}

Cancelling a common factor can change the separation property.
For example, \cref{cor:normalized-linear} gives
\(\cP(x(2x+1),x(3x+1))\).  But neither \(2x+1\) nor \(3x+1\)
is intersective, so \cref{thm:coprime-classification} shows that
\(\cP(2x+1,3x+1)\) fails.  Thus the coprime theorem alone does not
settle the degree-two case.

Berend and Bilu \cite{BerendBilu} gave an effective test for
intersectivity and a formula for the density of primes at which a
polynomial has a root.  Further results appear in
\cite{ElsholtzKlahnTechnau,Mishra2024}.  We use both roots modulo
primes and divisibility by prime powers.  The general classification
for pairs with common factors remains open.

The paper is organized as follows. \Cref{sec:local} gives some elementary obstructions.  \Cref{sec:sieve}
gives the two bounds used in \cref{sec:positive} to prove the main
results for coprime pairs.  \Cref{sec:common} gives intersectivity tests
and a necessary condition for pairs with common factors.
\Cref{sec:boundary} proves the results for three linear factors and
for pairs of degree at most two.

\section{Intersectivity and elementary obstructions}\label{sec:local}

\begin{definition}
A polynomial \(H\in\Z[x]\) is \emph{intersective} if, for every \(m\ge1\),
there is an \(n\in\Z\) such that \(H(n)\equiv0\pmod m\).
\end{definition}

For a nonzero \(P\in\Z[x]\) and \(m\ge1\), write
\[
  Z_P(m)=\{x\bmod m:P(x)\equiv0\pmod m\}.
\]
For a prime \(p\), write \(v_p\) for the usual \(p\)-adic valuation, with
\(v_p(0)=\infty\).

A polynomial is intersective if and only if it has a root in \(\Z_p\) for
every prime \(p\).  The next lemma shows that removing the content and
repeated irreducible factors does not change intersectivity.

\begin{lemma}
\label[lemma]{lem:radical-intersective}
Let \(P\in\Z[x]\) be nonzero.  Then \(P\) is intersective if and only if
its squarefree radical over \(\Q\), written as a primitive polynomial in
\(\Z[x]\), is intersective.
\end{lemma}

\begin{proof}
If \(P=c\ne0\) is constant, then neither \(P\) nor \(\radQ(P)=1\)
is intersective: \(P\) has no root modulo \(|c|+1\).
For nonconstant \(P\), write
\[
  P=c\prod_{i=1}^r f_i^{e_i},
  \qquad c\in\Z\setminus\{0\},
\]
where the \(f_i\) are distinct primitive irreducible polynomials and
\(e_i\ge1\).  Fix a prime \(p\).  The ring \(\Z_p\) is an integral domain.
In particular, the nonzero constant \(c\) is not a zero divisor, even when
\(p\mid c\).  Thus, for every \(z\in\Z_p\),
\[
  P(z)=0
  \quad\Longleftrightarrow\quad
  f_i(z)=0\text{ for some }i
  \quad\Longleftrightarrow\quad
  \prod_{i=1}^r f_i(z)=0.
\]
So removing the content and repeated factors does not change the roots in
\(\Z_p\).

If \(P\) has a root modulo every power of \(p\), then the sets
\[
  E_k=\{z\in\Z_p:P(z)\in p^k\Z_p\}\qquad(k\ge1)
\]
are nonempty, closed, and nested.  Since \(\Z_p\) is compact, their
intersection contains a zero of \(P\).  Conversely, a zero in \(\Z_p\)
gives a root modulo every power of \(p\).  Apply this at every prime and
use the Chinese remainder theorem to finish the proof.
\end{proof}

The following criterion for quadratics is known; see
\cite[Proposition 1 and Appendix B.1]{HamelLyallRice}.  We include the short
proof to fix the normalization used below.

\begin{proposition}\label[proposition]{prop:quadratic}
Let
\[
  P(x)=ax^2+bx+c\in\Z[x],\qquad a\ne0.
\]
Then \(P\) is intersective if and only if its discriminant is a square and
its rational roots, written in lowest terms as
\[
  r_1=\frac{u_1}{v_1},\qquad r_2=\frac{u_2}{v_2},
  \qquad v_1,v_2>0,
\]
satisfy \(\gcd(v_1,v_2)=1\).
\end{proposition}

\begin{proof}
Suppose the discriminant is not a square.  Its nontrivial quadratic
character is negative at infinitely many primes.  At each such prime
not dividing \(2a\) or the discriminant, \(P\) has no root.  So an
intersective quadratic factors over \(\Q\).

By Gauss's lemma, write
\[
  P(x)=K(A_1x+B_1)(A_2x+B_2)
\]
with \(K\in\Z\setminus\{0\}\) and
\(\gcd(A_i,B_i)=1\).  If a prime \(p\) divides both \(A_1\) and \(A_2\),
then both linear factors are \(p\)-adic units for every integer \(x\).
Thus \(v_p(P(x))=v_p(K)\) for all \(x\).  In particular, \(P\) has no
root modulo \(p^{v_p(K)+1}\).

Conversely, suppose \(\gcd(A_1,A_2)=1\).  For every prime power \(p^k\),
at least one \(A_i\) is a unit modulo \(p\).  This linear factor has a
root modulo \(p^k\).  The Chinese remainder theorem then gives a root
modulo every positive integer.  The reduced root denominators are
\(|A_1|\) and \(|A_2|\).
\end{proof}

\begin{proposition}
\label[proposition]{prop:both-nonintersective}
Let \(F,G\in\Z[x]\setminus\{0\}\).  If neither polynomial is intersective,
then \(\cP(F,G)\) does not hold.
\end{proposition}

\begin{proof}
Choose \(q,r\ge2\) such that \(F\) has no root modulo \(q\) and \(G\) has
no root modulo \(r\).  Put \(M=\lcm(q,r)\) and
\[
  A=\{M^k:k\ge1\}.
\]
For every integer \(n\), no member of \(A\) divides \(F(n)\) or \(G(n)\).
Also, neither polynomial has an integer zero.  Thus
\(\delta_A(F(n),G(n))=0\) for every \(n\).
\end{proof}

There is another obstruction when the two polynomials have the same roots
for infinitely many moduli.

\begin{proposition}
\label[proposition]{prop:identical-roots}
Let \(F,G\in\Z[x]\setminus\{0\}\).  Suppose there is an infinite set
\(A\subset\N\) such that
\[
  Z_F(a)=Z_G(a)\qquad(a\in A).
\]
Then \(\cP(F,G)\) does not hold.
\end{proposition}

\begin{proof}
Fix \(n\in\Z\).  By the assumption, for every \(a\in A\),
\[
  a\mid F(n)\quad\Longleftrightarrow\quad a\mid G(n).
\]
If \(F(n)\) and \(G(n)\) are both nonzero, they have exactly the same
divisors in \(A\).  Both counts are finite, so
\(d_A(F(n))=d_A(G(n))\).
If \(F(n)=0\), then every \(a\in A\) divides \(G(n)\).  A nonzero
integer has only finitely many positive divisors, so \(G(n)=0\).
The same argument applies when \(G(n)=0\), with \(F,G\) exchanged.
Thus both integers are nonzero and have equal divisor counts, or both
are zero.  In both cases, \eqref{eq:extended-difference} gives
\(\delta_A(F(n),G(n))=0\).  This holds for every \(n\), so the property fails.
\end{proof}

\begin{remark}
\label[remark]{rem:root-set-levels}
Distinct polynomials can have the same roots for infinitely many moduli
but different roots for other moduli.  For example, let
\(F(x)=x\) and \(G(x)=x^2\).  For every prime \(p\),
\[
  Z_F(p)=Z_G(p)=\{0\bmod p\},
\]
so \cref{prop:identical-roots} applies with \(A\) equal to the set of primes.
But for \(k\ge2\),
\[
  Z_F(p^k)=\{0\bmod p^k\},\qquad
  Z_G(p^k)=
  \{x\bmod p^k:p^{\lceil k/2\rceil}\mid x\},
\]
and these sets are different.  The assumption in
\cref{prop:identical-roots} concerns only the chosen moduli.  The
polynomials need not be associates, and their root sets need not agree
modulo every prime power.
\end{remark}

\section{Uniform almost-prime values on root progressions}
\label{sec:sieve}

For a nonzero integer \(m\), let \(\Omega(m)\) be the number of prime
factors of \(|m|\), counted with multiplicity.  If
\(Q\in\Q[x]\) is nonzero and integer-valued, define its fixed divisor by
\[
  \fd(Q)=\gcd\{Q(n):n\in\Z\},
\]
taking the positive gcd.  We write \(\tau(m)\) for the number of positive
divisors of a positive integer \(m\), and set \(\Omega(\pm1)=0\).

We use the following form of the classical saturation theorem.  Its bound
is the first of the two bounds in
\cite[Theorem~10.11, pp.~310--311]{HalberstamRichert}.

\begin{hrsaturation}
Let \(f_1,\ldots,f_g\in\Z[x]\) be distinct irreducible polynomials with
positive leading coefficients.  Put
\[
  f=\prod_{i=1}^g f_i,\qquad E=\deg f,
  \qquad \rho(p)=\#\{a\bmod p:f(a)\equiv0\pmod p\}.
\]
Assume that \(\rho(p)<p\) for every prime \(p\).  If the integer \(b\)
satisfies
\begin{equation}\label{eq:HR-bound}
  b>E-1+g\sum_{j=1}^g\frac1j
      +g\log\left(\frac{2E}{g}+\frac1{g+1}\right),
\end{equation}
then there is a constant \(\delta=\delta(b,f)>0\) such that, as
\(x\to\infty\),
\[
  \#\{1\le n\le x:f(n)=P_b\}
  \ge \delta\frac{x}{(\log x)^g}
      \left(1+O_f\left(\frac1{\log\log x}\right)\right).
\]
Here \(P_b\) denotes a positive integer with at most \(b\) prime factors,
counted with multiplicity.  In particular, \(\Omega(f(n))\le b\) for
infinitely many positive integers \(n\).
\end{hrsaturation}

The bound in \eqref{eq:HR-bound} depends only on the number and degrees
of the irreducible factors.  The density constant and the starting point
may depend on their coefficients.  Below, only the bound on the number
of prime factors needs to be uniform.

\begin{lemma}
\label[lemma]{lem:uniform-saturation}
For each integer \(D\ge1\), there is an integer \(R_D\) with the following
property.
If \(Q\in\Q[x]\) is a nonconstant integer-valued polynomial of degree at
most \(D\) and \(\fd(Q)=1\), then there are infinitely many positive
integers \(y\) such that
\[
  Q(y)\ne0,\qquad \Omega(Q(y))\le R_D.
\]
One possible choice is
\[
  R_D=Db_D,\qquad
  b_D=1+\left\lfloor
  D-1+D\sum_{j=1}^D\frac1j+D\log(2D+1)
  \right\rfloor.
\]
\end{lemma}

\begin{proof}
We apply the saturation theorem above.  Let \(d=\deg Q\le D\).
Choose a positive integer \(c\) such that
\[
  P(x)=cQ(x)\in\Z[x].
\]
Since \(\fd(Q)=1\), we have \(\fd(P)=c\).  For each prime \(p\mid c\),
choose \(a_p\in\Z\) such that
\[
  v_p(P(a_p))=v_p(c).
\]
By the Chinese remainder theorem, choose \(a\) satisfying
\[
  a\equiv a_p\pmod {p^{v_p(c)+1}}\qquad(p\mid c).
\]
We claim that
\[
  q(z)=Q(a+cz)=\frac{P(a+cz)}{c}
\]
lies in \(\Z[z]\) and has no fixed prime divisor.

Every coefficient of \(P(a+cz)-P(a)\) is divisible by \(c\).  Also,
\(P(a)/c=Q(a)\) is an integer.  Thus \(q\in\Z[z]\).

If \(p\mid c\), the choice of \(a\) gives
\[
  q(0)=\frac{P(a)}c\not\equiv0\pmod p.
\]
Now let \(p\nmid c\).  Suppose \(p\) divides \(q(z)\) for every integer
\(z\).  Since multiplication by \(c\) permutes the residue classes modulo
\(p\), this would imply that \(p\) divides \(P(n)\) for every integer
\(n\).  This contradicts \(\fd(P)=c\) and proves the claim.

Factor
\[
  q(z)=\pm\prod_{i=1}^r f_i(z)^{e_i},
\]
where the \(f_i\) are distinct primitive irreducible polynomials with
positive leading coefficients.  The content is one because \(q\) has no fixed prime divisor.  The squarefree product
\[
  q_{\mathrm{sf}}(z)=\prod_{i=1}^r f_i(z)
\]
also has no fixed prime divisor.

For a prime \(p\), let \(\rho(p)\) be the number of roots of
\(q_{\mathrm{sf}}\) modulo \(p\).  Since there is no fixed prime divisor,
\(\rho(p)<p\) for every \(p\).  For each integer \(b\) satisfying the bound
in the saturation theorem, there are infinitely many positive \(z\) such that
\[
  \Omega(q_{\mathrm{sf}}(z))\le b.
\]
The integer \(q_{\mathrm{sf}}(z)\) is positive for all sufficiently large
positive \(z\), as the saturation theorem requires.  Since
\(\sum_i e_i\deg f_i=d\), we have \(e_i\le d\).  Thus
\[
  \Omega(q(z))
  =\sum_{i=1}^r e_i\Omega(f_i(z))
  \le d\,\Omega(q_{\mathrm{sf}}(z)).
\]
Let \(E=\deg q_{\mathrm{sf}}\le D\) and \(r\le D\).  The theorem applies
when the integer \(b\) is greater than
\[
  E-1+r\sum_{j=1}^r\frac1j+
  r\log\left(\frac{2E}{r}+\frac1{r+1}\right).
\]
This expression is less than \(b_D\), so we may take \(b=b_D\).  We get
\[
  \Omega(q(z))\le d b_D\le D b_D=R_D.
\]
The integers \(y=a+cz\) are distinct and positive for all sufficiently
large positive \(z\).  This proves the lemma.
\end{proof}

\begin{remark}
The starting point for these infinitely many \(y\) may depend on the
coefficients of \(Q\).  Only the bound \(R_D\) is uniform.  This is all
we need below, where the polynomial on the progression changes with the
number of prescribed divisors.
\end{remark}

\begin{lemma}
\label[lemma]{lem:fixed-divisor}
Let \(H,K\in\Z[x]\setminus\{0\}\) be coprime in \(\Q[x]\).  There is a
constant \(C_0=C_0(H,K)\) such that, for all \(a\in\Z\) and \(M\ge1\)
satisfying \(M\mid H(a)\),
\[
  \fd\bigl(K(a+M\,\cdot)\bigr)\le C_0.
\]
\end{lemma}

\begin{proof}
There are \(B,C\in\Z[x]\) and a nonzero integer \(R\) such that
\begin{equation}\label{eq:bezout}
  B(x)H(x)+C(x)K(x)=R.
\end{equation}
Fix \(a,M\) with \(M\mid H(a)\), and let
\[
  \Delta=\fd\bigl(K(a+M\,\cdot)\bigr).
\]
For a prime \(p\), write
\[
  \delta_p=v_p(\Delta)
  =\min_{y\in\Z}v_p(K(a+My)).
\]

If \(p\nmid M\), the progression \(a+My\) meets every residue class
modulo every power of \(p\).  Thus
\[
  \delta_p\le v_p(\fd(K)).
\]

Suppose \(p\mid M\), and set \(m=v_p(M)\).  Since \(M\mid H(a)\),
we have \(p^m\mid H(a+My)\) for every \(y\).  By \eqref{eq:bezout},
\[
  p^{\min(m,\delta_p)}\mid R.
\]
If \(m>v_p(R)\), we get \(\delta_p\le v_p(R)\).  Now suppose
\(1\le m\le v_p(R)\).  Then \(p\mid R\).  There are only finitely many
choices of \(m\) and a residue class \(a\bmod p^m\) with
\(H(a)\equiv0\pmod {p^m}\).  For each choice,
\[
  e_{p,m,a}:=\min_{x\equiv a\;(\mathrm{mod}\;p^m)}v_p(K(x))
\]
is finite.  The residue class contains infinitely many integers, and
\(K\) has only finitely many integer zeros.  This minimum equals
\(\delta_p\).  To see this, write \(M=p^mM'\) with \(p\nmid M'\).
Choose an \(x\) that attains the minimum.  Since \(M'\) is a \(p\)-adic
unit, we can choose \(y\) such that
\[
  a+My\equiv x\pmod {p^{e_{p,m,a}+1}}.
\]
Then \(v_p(K(a+My))=e_{p,m,a}\).  The maximum over these finitely many
cases bounds \(\delta_p\).  Only primes dividing \(R\,\fd(K)\) can have
\(\delta_p>0\).  Combining their bounds gives a constant \(C_0(H,K)\).
\end{proof}

\section{Simultaneous dominance and the coprime classification}
\label{sec:positive}

\begin{proof}[Proof of \cref{thm:simultaneous-dominance}]
We follow the argument with arithmetic progressions and a sieve in
\cite[proof of Theorem~4]{DingProblem26}.  We use the estimates from
\cref{sec:sieve}.  Put
\[
  K(x)=\prod_{j=1}^sG_j(x).
\]
First suppose that \(K\) is nonconstant.  Let
\[
  A_t=\{a_1,\ldots,a_t\}\subset\N,\qquad
  M_t=\prod_{i=1}^t a_i.
\]
If \(A\) is finite, take \(A_t=A\).  If \(A\) is infinite, list its
elements in increasing order and let \(A_t\) contain the first \(t\).

Choose \(u_t\in\Z\) such that \(M_t\mid H(u_t)\).  Then for every
\(y\in\Z\),
\[
  M_t\mid H(u_t+M_ty)\mid F(u_t+M_ty).
\]
Thus \(a_i\mid F(u_t+M_ty)\) for every \(1\le i\le t\) and every
integer \(y\).

Let
\[
  \Delta_t=\fd\bigl(K(u_t+M_t\,\cdot)\bigr).
\]
By \cref{lem:fixed-divisor}, \(1\le\Delta_t\le C_0(H,K)\).  The polynomial
\[
  Q_t(y)=\frac{K(u_t+M_ty)}{\Delta_t}
\]
is nonconstant and integer-valued.  It has degree \(\deg K\) and fixed
divisor one.  By \cref{lem:uniform-saturation}, there are arbitrarily
large positive \(y\) such that
\[
  Q_t(y)\ne0,\qquad \Omega(Q_t(y))\le R_{\deg K}.
\]
For \(n=u_t+M_ty\), all \(G_j(n)\) are nonzero, and
\[
  \Omega(K(n))
  \le R_{\deg K}+
  \max_{1\le m\le C_0(H,K)}\Omega(m)=:B,
\]
so, for every \(j\),
\begin{equation}\label{eq:comparison-count-bound}
  d_A(G_j(n))
  \le\tau(|G_j(n)|)
  \le\tau(|K(n)|)
  \le2^B=:C.
\end{equation}
The constant does not depend on \(t\) or \(A\).

If \(A\) is finite and \(|A|>C\), use \(A_t=A\) to obtain
\[
  d_A(F(n))\ge|A|>C\ge\max_jd_A(G_j(n))
\]
for infinitely many \(n\).  If \(A\) is infinite, at each stage choose a
large enough \(y\) from \cref{lem:uniform-saturation} so that the integers
\(n=u_t+M_ty\) increase to infinity.  Then
\[
  d_A(F(n))-\max_jd_A(G_j(n))\ge t-C,
\]
which proves the second part.

If \(K\) is a nonzero constant, every \(G_j\) is a nonzero constant.
Set \(C=\tau(|K|)\).  Then \(d_A(G_j(n))\le C\) for every \(j,n\).
The same root progressions give \(d_A(F(n))\ge t\) for arbitrarily
large \(n\).  The same argument proves both parts.
\end{proof}

\begin{proof}[Proof of \cref{thm:coprime-classification}]
If neither \(F\) nor \(G\) is intersective, the property fails by
\cref{prop:both-nonintersective}.  Conversely, suppose \(F\) is
intersective.  Then it is nonconstant.  Apply
\cref{thm:simultaneous-dominance} with \(H=F\), \(s=1\), and \(G_1=G\).
The difference \(d_A(F(n))-d_A(G(n))\) is unbounded above, so
\(\cP(F,G)\) holds.  If \(G\) is intersective, exchange \(F\) and \(G\).
\end{proof}

\begin{corollary}\label[corollary]{cor:oscillation}
Let \(F,G\in\Z[x]\setminus\{0\}\) be coprime in \(\Q[x]\), and suppose both
are intersective.  For every infinite \(A\subset\N\),
\[
  \limsup_{n\to\infty}\bigl(d_A(F(n))-d_A(G(n))\bigr)=\infty
\]
and
\[
  \liminf_{n\to\infty}\bigl(d_A(F(n))-d_A(G(n))\bigr)=-\infty,
\]
where we may restrict both sequences to integers for which
\(F(n)\) and \(G(n)\) are nonzero.
\end{corollary}

\begin{proof}
Apply \cref{thm:simultaneous-dominance} first with \(H=F\) and
\(G_1=G\), then with \(F,G\) exchanged.
\end{proof}

The cases with a zero or constant polynomial are immediate.

\begin{corollary}
\label[corollary]{cor:zero-constant}
Let \(F,G\in\Z[x]\).
\begin{enumerate}
  \item If exactly one of \(F,G\) is the zero polynomial, then
  \(\cP(F,G)\) holds.
  \item If both are zero, then \(\cP(F,G)\) fails.
  \item If \(C\ne0\) is constant and \(P\) is nonzero and nonconstant, then
  \(\cP(C,P)\) holds if and only if \(P\) is intersective.
\end{enumerate}
\end{corollary}

\begin{proof}
The first two parts follow from \eqref{eq:extended-difference}, since a
nonzero polynomial has only finitely many integer zeros.  The third part
follows from the coprime classification.
\end{proof}

\section{Intersectivity tests and pairs with common factors}
\label{sec:common}

Intersectivity has a simple definition.  To test it, one uses a Galois
condition and finitely many congruences.  Replace the polynomial by its
primitive squarefree radical and factor it over \(\Q\) as
\[
  R=h_1\cdots h_r.
\]
Let \(L\) be the splitting field, let \(\Gamma=\Gal(L/\Q)\), and let
\(H_i\) be the stabilizer of a root of \(h_i\).  Then \(R\) has a root
modulo all but finitely many primes exactly when
\[
  \Gamma=\bigcup_{i=1}^r\ \bigcup_{\sigma\in\Gamma}
  \sigma H_i\sigma^{-1}.
\]
Berend and Bilu \cite{BerendBilu} proved that \(R\) is intersective exactly
when this condition holds and \(R\) has a root modulo one further integer.
That integer can be computed explicitly, and its prime factors divide the
resultants \(\Res(h_i,h_i')\).  This gives a finite test.  The test may
require factoring over \(\Q\), finding a splitting field and its
permutation group, and solving finitely many prime-power congruences.
We give the condition at almost all primes the following name.

\begin{definition}
A nonzero polynomial \(R\in\Z[x]\) is \emph{prime-covering} if it has a
root modulo all but finitely many primes.
\end{definition}

Every nonzero intersective polynomial is prime-covering.  A squarefree
prime-covering polynomial has a simple root modulo \(p\) for all but
finitely many primes.  The excluded primes include those dividing its
leading coefficient or discriminant.  Hensel's lemma lifts each simple
root to \(\Z_p\).  To be intersective, the polynomial must also have a
root in \(\Z_p\) at each excluded prime.  See \cref{ex:236} for an
example of this difference.  For tests and constructions with given
factorizations, see
\cite{ElsholtzKlahnTechnau,MishraQuadratic,MishraMinimal,Mishra2024}.

In degree at most three, we have the following explicit test.  We use
this standard form of Hensel's lemma; see
\cite[Lemma~12.2.17]{Voight}.

\begin{stronghensel}
Let \(p\) be a prime, \(f\in\Z_p[x]\), and \(a\in\Z_p\). If
\[
  v_p(f(a))>2v_p(f'(a)),
\]
then \(f\) has a zero in \(\Z_p\).
\end{stronghensel}

\begin{proposition}
\label[proposition]{prop:cubic-intersectivity}
Let \(P\in\Z[x]\setminus\{0\}\) have degree at most three, and let \(R\)
be its squarefree radical over \(\Q\), represented primitively in \(\Z[x]\).
Then \(P\) is intersective if and only if one of the following conditions
holds.
\begin{enumerate}
  \item The polynomial \(R\) splits over \(\Q\).  On writing
  \[
    R(x)=\pm\prod_{i=1}^r(a_i x+b_i),
    \qquad 1\le r\le3,
  \]
  with each factor primitive, one has
  \[
    \gcd(|a_1|,\ldots,|a_r|)=1.
  \]
  \item The polynomial \(R\) has degree three and
  \[
    R(x)=\pm L(x)Q(x),\qquad L(x)=ax+b,
  \]
  where \(L\) is primitive and \(Q\in\Z[x]\) is primitive and irreducible
  quadratic.  For every prime \(p\mid a\), the polynomial \(Q\) has a root
  in \(\Z_p\).
\end{enumerate}
The test in \textup{(ii)} is finite.  If
\[
  s_p=v_p\bigl(\lvert\Res(Q,Q')\rvert\bigr),
\]
then its local condition holds exactly when
\[
  Q(x)\equiv0\pmod {p^{2s_p+1}}
\]
has a solution for every prime \(p\mid a\). 
\end{proposition}

\begin{proof}
If \(P\) is constant, then \(R=1\).  Neither condition holds, and \(P\)
is not intersective.  Now suppose that \(P\) is nonconstant.  By
\cref{lem:radical-intersective}, we only need to consider \(R\).
A primitive linear polynomial \(ax+b\) has a root in \(\Z_p\) exactly
when \(p\nmid a\).  Since \(\Z_p\) is an integral domain, a product of
primitive linear factors has a root in \(\Z_p\) exactly when at least one
leading coefficient is a \(p\)-adic unit.  This holds for every prime
exactly when the gcd condition in \textup{(i)} holds.

Next suppose that \(R=LQ\) as in \textup{(ii)}.  If \(p\nmid a\), then
\(L\) has a root in \(\Z_p\).  If \(p\mid a\), then \(p\nmid b\) because
\(L\) is primitive.  Thus \(L(x)\) is a \(p\)-adic unit for every
\(x\in\Z_p\).  At such a prime, \(R\) has a \(p\)-adic root exactly when
\(Q\) does.  This proves the local criterion.

To prove the finite congruence test, put
\(S=\Res(Q,Q')\ne0\).  There are \(A,B\in\Z[x]\) such that
\[
  A(x)Q(x)+B(x)Q'(x)=S.
\]
If \(Q(x)\equiv0\pmod {p^{2s_p+1}}\), the identity implies
\(v_p(Q'(x))\le s_p\).  Thus,
\[
  v_p(Q(x))
  \ge 2s_p+1
  >2s_p
  \ge 2v_p(Q'(x)).
\]
The strong Hensel lemma, applied with \(f=Q\) and \(a=x\), gives a
zero of \(Q\) in \(\Z_p\).
The converse follows by reducing a \(p\)-adic zero modulo
\(p^{2s_p+1}\).

We now rule out the other factorizations.  A squarefree polynomial of
degree at most three with no linear factor is irreducible of degree two
or three.  Its nontrivial transitive Galois action contains a derangement.
Chebotarev then gives infinitely many primes modulo which it has no root.
Thus it is not prime-covering and is not intersective.  The two listed
cases cover all other factorizations.
\end{proof}

\begin{example}[A cubic intersective polynomial without an integer root]
\label[example]{ex:cubic-intersective}
The polynomial
\[
  P(x)=(2x+1)(x^2+x+2)
\]
has no integer root.  The quadratic factor \(Q(x)=x^2+x+2\) is irreducible
over \(\Q\).  The only prime dividing the leading coefficient of
\(2x+1\) is \(2\).
Here \(v_2(Q(0))=1>2v_2(Q'(0))=0\), so the strong Hensel lemma
gives a zero of \(Q\) in \(\Z_2\). \Cref{prop:cubic-intersectivity}\textup{(ii)}
therefore shows that \(P\) is intersective.
\end{example}

This proposition tests intersectivity for one polynomial of degree at most
three.  We do not classify \(\cP(F,G)\) for all cubic pairs with common
factors; see \cref{q:exceptional-boundary}.

We use the following notation for pairs with common factors.  For a nonzero
\(P\in\Z[x]\), write
\[
  P=c_P\prod_{\phi\in\cI_P}\phi^{e_P(\phi)},
\]
where \(c_P\in\Z\setminus\{0\}\) and the \(\phi\) are pairwise
nonassociate primitive irreducible polynomials with positive leading
coefficients.  For \(F,G\ne0\), set
\[
  \cI_C=\cI_F\cap\cI_G,\qquad
  \cI_U=\cI_F\setminus\cI_G,\qquad
  \cI_V=\cI_G\setminus\cI_F,
\]
and define
\[
  C=\prod_{\phi\in\cI_C}\phi,\qquad
  U=\prod_{\phi\in\cI_U}\phi,\qquad
  V=\prod_{\phi\in\cI_V}\phi.
\]
In the products defining \(C,U,V\), each irreducible factor
is included only once.  Its multiplicity in \(F\) or \(G\)
does not affect these products. An empty product is \(1\).  Then \(C,U,V\) are pairwise coprime and
\[
  \radQ(F)=CU,\qquad \radQ(G)=CV
\]
with the positive-leading-coefficient normalization chosen above.

For a nonconstant squarefree polynomial, prime-covering is a Galois
condition.  A monic prime-covering polynomial in \(\Z[x]\) with no integer
root is called \emph{exceptional}; see
\cite{ElsholtzKlahnTechnau}.

\begin{proof}[Proof of \cref{thm:common-support}]
Suppose that \(UV\) has no root modulo any prime in an infinite set
\(S\).  Delete from \(S\) the finitely many primes dividing \(c_Fc_G\).
For \(p\in S\), neither \(U\) nor \(V\) has a root modulo \(p\).  Therefore,
for every integer \(n\),
\[
  p\mid F(n)
  \quad\Longleftrightarrow\quad
  p\mid C(n)
  \quad\Longleftrightarrow\quad
  p\mid G(n).
\]
Take \(A=S\) as an infinite set of positive integers.  Then
\[
  Z_F(p)=Z_G(p)\qquad(p\in A).
\]
\Cref{prop:identical-roots} proves the result.
\end{proof}

We next recall how Chebotarev relates primes to permutations of roots.
Let \(K/\Q\) be a finite Galois extension.  Choose a conjugacy class
\(\mathcal C\) in \(\Gal(K/\Q)\).  Consider the unramified primes with
Frobenius class \(\mathcal C\).  Their natural density is
\(|\mathcal C|/|\Gal(K/\Q)|\); see
\cite[Section~2, Chebotarev theorem]{ElsholtzKlahnTechnau}.  Suppose
that \(K\) is the splitting field of a squarefree polynomial \(R\).
For primes not dividing its leading coefficient or discriminant, the
cycle type of Frobenius is the factorization type of \(R\bmod p\).
Thus \(R\bmod p\) has a linear factor exactly when Frobenius fixes a root.
A permutation with no fixed root is called a \emph{derangement}.
Outside the finite excluded set, \(R\) has no root modulo \(p\) exactly
when its Frobenius class consists of derangements.  The resulting density
formula is stated in
\cite[Theorem~2, p.~1666]{BerendBilu} and
\cite[Theorem~2.1]{ElsholtzKlahnTechnau}.

\begin{proposition}
\label[proposition]{prop:frobenius}
Let \(R\in\Z[x]\) be nonconstant and squarefree, let \(L\) be its splitting
field, and let \(\Gamma=\Gal(L/\Q)\) act on the roots of \(R\).  The set of
primes modulo which \(R\) has a root has natural density
\[
  \frac{\#\{\sigma\in\Gamma:\sigma\text{ fixes a root of }R\}}
       {|\Gamma|}.
\]
\end{proposition}

\begin{proof}
The required primes have Frobenius classes whose elements fix a root.
Adding the densities of these classes gives the formula.  Removing finitely
many bad primes does not change the density.
\end{proof}

\begin{proof}[Proof of \cref{cor:derangement}]
If \(\sigma\in\Gamma\) fixes no root, then no element of its conjugacy
class fixes a root.  By \cref{prop:frobenius}, the primes modulo which
\(UV\) has no root have positive density.  Apply \cref{thm:common-support}.
\end{proof}

\begin{corollary}
\label[corollary]{cor:factor-conditions}
Let \(F,G\ne0\), and let \(C,U,V\) be defined as above.
\begin{enumerate}
  \item If \(U\) is nonconstant and intersective, or if \(V\) is
  nonconstant and intersective, then \(\cP(F,G)\) holds.
  \item If \(\cP(F,G)\) holds, then \(UV\) is prime-covering; equivalently,
  when \(UV\) is nonconstant, its Galois action has no derangement.
\end{enumerate}
\end{corollary}

\begin{proof}
The polynomial \(U\) divides \(F\) in \(\Z[x]\) and is coprime to \(G\).
Likewise, \(V\) divides \(G\) and is coprime to \(F\).  The first claim
now follows from \cref{thm:simultaneous-dominance}.  The second follows
from \cref{thm:common-support,cor:derangement}.  To prove the equivalence,
suppose that the Galois action has no derangement.  Then every Frobenius
element fixes a root.  Thus \(UV\) has a root modulo every prime outside
the finite excluded set.
\end{proof}

When \(F\) and \(G\) are coprime, we have \(C=1\).  In this case,
\cref{prop:both-nonintersective} gives the converse to the sufficient
condition in \cref{cor:factor-conditions}.  This proves
\cref{thm:coprime-classification}.  When there are common factors, both
\(F\) and \(G\) can be intersective even if neither \(U\) nor \(V\) is
intersective.

\section{Prime-power conditions and the degree-two classification}
\label{sec:boundary}

Prime-covering is weaker than intersectivity, as the following classical
example shows.

\begin{example}[Prime-covering but not intersective]\label[example]{ex:236}
Let
\[
  E(x)=(x^2-2)(x^2-3)(x^2-6).
\]
For every odd prime \(p\), if neither \(2\) nor \(3\) is a quadratic residue
modulo \(p\), then \(6\) is a quadratic residue.  At \(p=2\) or \(p=3\),
the residue \(x=0\) is a root, so \(E\) has a root modulo every prime.

If \(x\) is even, then
\[
  E(x)\equiv4\pmod8,
\]
while if \(x\) is odd then
\[
  E(x)\equiv6\pmod8.
\]
Thus \(E\) has no root modulo \(8\) and is not intersective.  See
\cite{ElsholtzKlahnTechnau} for this example.
\end{example}

An intersective polynomial without a rational root is
\[
  (x^2-13)(x^2-17)(x^2-221);
\]
see \cite[Example 1]{BerendBilu}.  These examples show that roots modulo
all but finitely many primes do not imply intersectivity.  In the proof
below, we also use valuations at primes dividing the leading coefficients
and the resultant of two linear factors.

We now prove the common-factor classification stated in the introduction.

\begin{proof}[Proof of \cref{thm:linear-defects}]
We first prove that the conditions are necessary.
Suppose that a prime \(p\) divides both \(u\) and \(v\).
Since \(U\) and \(V\) are primitive, \(U(n)\) and \(V(n)\) are
\(p\)-adic units for every \(n\).  Take the set
\[
  \mathcal A=\{p^{3^j}:j\ge1\},
\]
and write
\[
  N(X)=\#\{j\ge1:3^j\le X\}\qquad(X\ge0).
\]
For all sufficiently large \(n\), the polynomial values are nonzero.
Writing \(\tau=v_p(L(n))\), the two divisor counts are
\[
  N\bigl(v_p(\gamma)+r_L\tau\bigr)
  \quad\text{and}\quad
  N\bigl(v_p(\delta)+s_L\tau\bigr).
\]
Their difference is bounded independently of \(\tau\), since
\[
  N(X)=\log_3(X+1)+O(1)
\]
and the ratio of the two arguments, after adding one to each, has fixed
positive upper and lower bounds.  Thus \(\cP(F,G)\) fails.  We have proved
\[
  \cP(F,G)\quad\Longrightarrow\quad \gcd(|u|,|v|)=1.
\]

By \cref{lem:radical-intersective,prop:quadratic},
\[
  F\text{ is intersective}
  \quad\Longleftrightarrow\quad
  \gcd(|\ell|,|u|)=1,
\]
and
\[
  G\text{ is intersective}
  \quad\Longleftrightarrow\quad
  \gcd(|\ell|,|v|)=1.
\]
If both gcds are greater than one, neither \(F\) nor \(G\) is intersective.
Then \(\cP(F,G)\) fails by \cref{prop:both-nonintersective}.  Thus both
stated conditions are necessary.

\medskip
We now prove that the conditions are sufficient.
By symmetry, suppose that
\[
  \gcd(|\ell|,|u|)=\gcd(|u|,|v|)=1.
\]
Put
\[
  R=\Res(L,U)=\ell\mu-u\lambda\ne0,
  \qquad
  h_q=v_q(R),
  \qquad
  \beta_q=v_q(\delta)
  \quad(q\mid v).
\]
We will use the identity
\begin{equation}\label{eq:linear-resultant-identity}
  uL(x)-\ell U(x)=-R.
\end{equation}

Let \(\mathcal A\subset\N\) be any infinite set.  Define
\begin{equation}\label{eq:A-small}
  \mathcal A_{\le}
  =
  \{m\in\mathcal A:
    v_q(m)\le\beta_q+s_Lh_q
    \text{ for every prime }q\mid v\}.
\end{equation}
When \(v=\pm1\), the condition in \eqref{eq:A-small} is empty and
\(\mathcal A_{\le}=\mathcal A\).

\smallskip
\noindent\emph{Case 1: \(\mathcal A_{\le}\) is infinite.}
Choose distinct \(m_1,\ldots,m_t\in\mathcal A_{\le}\), put
\[
  r_i=\frac{m_i}{\gcd(m_i,|\delta|)},
\]
and write
\[
  d_i=\prod_{q\mid v}q^{v_q(r_i)},
  \qquad
  c_i=\frac{r_i}{d_i}.
\]
Define
\[
  D=\lcm(d_1,\ldots,d_t),
  \qquad
  E=\lcm(c_1,\ldots,c_t),
  \qquad
  B=\prod_{q\mid v}q^{s_Lh_q}.
\]
For every \(q\mid v\) and every \(i\), we have
\(m_i\in\mathcal A_{\le}\), so
\[
  v_q(d_i)
  =v_q(r_i)
  =\max\{v_q(m_i)-\beta_q,0\}
  \le s_Lh_q.
\]
Thus,
\[
  D\mid B,
  \qquad
  \gcd(D,E)=1,
  \qquad
  \gcd(v,E)=1.
\]

We also have \(\gcd(\ell,D)=1\).  To prove this, suppose that a prime
\(q\) divides both \(\ell\) and \(D\).  By the definition of \(D\),
we have \(q\mid v\) and \(v_q(r_i)>0\) for some \(i\).  Hence
\[
  1\le v_q(r_i)
    =v_q(m_i)-\beta_q
    \le s_Lh_q.
\]
Thus \(h_q\ge1\), so \(q\mid R\).  But \(\gcd(u,v)=1\) gives
\(q\nmid u\).  Also, \(q\nmid\lambda\) because \(L(x)=\ell x+\lambda\)
is primitive and \(q\mid\ell\).  Thus
\[
  R=\ell\mu-u\lambda
    \equiv-u\lambda\not\equiv0\pmod q,
\]
a contradiction.

The Chinese remainder theorem now gives \(n_0\in\Z\) such that
\[
  L(n_0)\equiv0\pmod D,
  \qquad
  V(n_0)\equiv0\pmod E.
\]
Set \(M=DE\).  For \(n=n_0+My\), we have
\[
  d_i\mid L(n),
  \qquad
  c_i\mid V(n),
\]
and hence
\[
  r_i=d_ic_i\mid L(n)V(n).
\]
Since \(m_i\mid\delta r_i\) and \(s_L,s_V\ge1\), every \(m_i\) divides
\[
  G(n)=\delta L(n)^{s_L}V(n)^{s_V}.
\]
Also, since \(E\mid V(n_0)\),
\[
  M=DE\mid D\,V(n_0).
\]

For every positive divisor \(d\) of \(B\), put \(H_d=dV\).
The polynomials \(H_d\) and
\[
  F=\gamma L^{r_L}U^{r_U}
\]
are coprime in \(\Q[x]\).  For each \(d\), let \(C_0(H_d,F)\) be the
bound in \cref{lem:fixed-divisor}.  Set
\[
  C_*=\max_{d\mid B}C_0(H_d,F).
\]
The set of divisors of \(B\) is finite.  Since \(D\mid B\) and
\(M\mid H_D(n_0)\), \cref{lem:fixed-divisor} gives
\begin{equation}\label{eq:uniform-fixed-divisor-linear-defects}
  \Delta
  :=
  \fd\bigl(F(n_0+M\,\cdot)\bigr)
  \le C_*.
\end{equation}

The polynomial
\[
  \frac{F(n_0+My)}{\Delta}
\]
is integer-valued, has degree \(r_L+r_U\), and has fixed divisor one.
Apply \cref{lem:uniform-saturation}.  Remove the finitely many integers
\(y\) for which \(G(n_0+My)=0\).  This leaves arbitrarily large
positive integers \(y\) such that
\[
  F(n_0+My)G(n_0+My)\ne0
\]
and
\[
  \Omega(F(n_0+My))
  \le
  R_{r_L+r_U}
  +\max_{1\le d\le C_*}\Omega(d)
  =:B_0.
\]
The constant \(B_0\) does not depend on \(t\) or on the chosen elements
\(m_1,\ldots,m_t\).  Every \(m_i\) divides \(G(n_0+My)\), and
\[
  d_{\mathcal A}(F(n_0+My))
  \le \tau(|F(n_0+My)|)
  \le 2^{B_0},
\]
we obtain
\[
  d_{\mathcal A}(G(n_0+My))
  -
  d_{\mathcal A}(F(n_0+My))
  \ge t-2^{B_0}.
\]
For each \(t\), choose \(y\) so that the integers \(n=n_0+My\)
increase to infinity.  This proves
\(\cP(F,G)\) in this case.

\smallskip
\noindent\emph{Case 2: \(\mathcal A_{\le}\) is finite.}
Choose distinct
\[
  m_1,\ldots,m_t\in
  \mathcal A\setminus\mathcal A_{\le},
\]
and write
\[
  d_i=\prod_{p\mid u}p^{v_p(m_i)},
  \qquad
  c_i=\frac{m_i}{d_i}.
\]
Put
\[
  D=\lcm(d_1,\ldots,d_t),
  \qquad
  E=\lcm\left(
      c_1,\ldots,c_t,
      \prod_{q\mid v}q^{h_q+1}
    \right).
\]
Every prime divisor of \(D\) divides \(u\), so
\(\gcd(\ell,D)=1\).  By the definition of \(c_i\), no prime divisor of
\(u\) divides any \(c_i\).  Also, \(\gcd(u,v)=1\).  Thus
\[
  \gcd(\ell,D)=\gcd(u,E)=\gcd(D,E)=1.
\]

The Chinese remainder theorem gives \(n_0\in\Z\) such that
\[
  L(n_0)\equiv0\pmod D,
  \qquad
  U(n_0)\equiv0\pmod E.
\]
Set \(M=DE\).  Along the progression \(n=n_0+My\), we have
\[
  d_i\mid L(n),
  \qquad
  c_i\mid U(n),
\]
and therefore
\[
  m_i=d_ic_i\mid L(n)U(n)\mid F(n)
\]
for every \(i\).

Fix a prime \(q\mid v\).  Since
\[
  q^{h_q+1}\mid E\mid U(n),
\]
equation \eqref{eq:linear-resultant-identity} gives
\[
  uL(n)\equiv-R\pmod{q^{h_q+1}}.
\]
Because \(v_q(R)=h_q\) and \(q\nmid u\), it follows that
\[
  v_q(L(n))=h_q.
\]
Since \(q\mid v\) and \(V(x)=vx+\nu\) is primitive, we also have
\(q\nmid\nu\).  Hence
\[
  V(n)\equiv\nu\not\equiv0\pmod q
  \quad\text{and}\quad
  v_q(V(n))=0.
\]
Thus,
\begin{equation}\label{eq:exact-G-valuation}
  v_q(G(n))=\beta_q+s_Lh_q
  \qquad(q\mid v).
\end{equation}

Every member of
\(\mathcal A\setminus\mathcal A_{\le}\) has \(q\)-adic valuation
greater than \(\beta_q+s_Lh_q\) for at least one prime \(q\mid v\).
By \eqref{eq:exact-G-valuation}, no such member divides \(G(n)\).
Choose positive representatives \(n_t\to\infty\) that avoid the finitely
many roots of \(FG\).  Then
\[
  d_{\mathcal A}(F(n_t))
  -
  d_{\mathcal A}(G(n_t))
  \ge t-|\mathcal A_{\le}|.
\]
Letting \(t\to\infty\) completes the proof.
\end{proof}

\begin{proof}[Proof of \cref{cor:degree-two}]
If \(C=1\), the polynomials are coprime.  By
\cref{thm:coprime-classification}, the property holds exactly when one of
\(F,G\) is intersective.  In this case, \(UV\) is also intersective.

Suppose that \(C\) is nonconstant.  If \(U=V=1\), let \(\mathcal A\) be
the set of all primes not dividing either polynomial content.  For every
\(p\in\mathcal A\), we have \(p\mid F(n)\) exactly when \(p\mid C(n)\),
and the same holds for \(G\).  Thus
\[
  Z_F(p)=Z_G(p).
\]
The property fails by \cref{prop:identical-roots}, while \(UV=1\) is not
intersective.

Next suppose that \(U=1\) and \(V\) is nonconstant.  Since
\(\deg C+\deg V\le\deg G\le2\), both \(C\) and \(V\) are linear.
Each factor occurs once in \(G\).  Every irreducible factor of \(F\) is
associated to \(C\).  Set \(L=C\).  Then there are nonzero integers
\(\gamma,\delta\) and \(r\in\{1,2\}\) such that
\[
  F=\gamma L^r,\qquad G=\delta LV.
\]
If \(V\) is intersective, apply \cref{thm:simultaneous-dominance} to the
factor \(V\mid G\).  This proves the property.  Otherwise, choose a prime
\(p\) dividing the leading coefficient of \(V\).  Since \(V\) is primitive,
every \(V(n)\) is a \(p\)-adic unit.  Take the set
\[
  \mathcal A=\{p^{3^j}:j\ge1\}.
\]
For \(X\ge0\), write
\[
  N(X)=\#\{j\ge1:3^j\le X\}.
\]
For every integer \(n\) such that \(L(n)\ne0\), set \(t=v_p(L(n))\).
The integers \(F(n),G(n)\) are then nonzero, and their divisor counts are
\[
  N\bigl(v_p(\gamma)+rt\bigr)
  \quad\text{and}\quad
  N\bigl(v_p(\delta)+t\bigr).
\]
Their difference is bounded independently of \(t\).  Indeed, the ratio of
the two arguments, after adding one to each, has fixed positive upper and
lower bounds, and \(N(X)=\log_3(X+1)+O(1)\).  A zero of \(L\) is a zero
of both polynomials, and \(V\) has no integer zero.  Thus the extended
difference is bounded, so the property fails.  The case where only \(U\)
is nonconstant follows by symmetry.

In the remaining case \(C,U,V\) are all nonconstant.  The inequalities
\[
  \deg C+\deg U\le\deg F\le2,
  \qquad
  \deg C+\deg V\le\deg G\le2
\]
show that all three pairwise coprime polynomials are linear.  Equality
holds in both inequalities, so every factor occurs once on each side
where it appears.  This is the case \(r_L=r_U=s_L=s_V=1\) of
\cref{thm:linear-defects}.

These cases cover all factorizations in degree at most two.  The criterion
uses only \(C,U,V\).  By \cref{lem:radical-intersective}, removing contents
and multiplicities does not change intersectivity.  Apply the same criterion
to \(\radQ(F),\radQ(G)\) to obtain the final equivalence.
\end{proof}

\begin{proof}[Proof of \cref{cor:normalized-linear}]
If \(a=b\), the two polynomials are identical.  Otherwise apply
\cref{thm:linear-defects} with \(L=x\), \(U=ax+1\), \(V=bx+1\), and all
four exponents equal to one.
\end{proof}

Suppose that \(\gcd(|a|,|b|)=1\) and \(|a|,|b|>1\).  The factors
\(ax+1\) and \(bx+1\) in \cref{cor:normalized-linear} are not intersective.
Their product is intersective: modulo each prime power, at least one of
\(a,b\) is invertible.  The Chinese remainder theorem then combines the
roots of the linear factors.  Thus the sufficient condition in
\cref{cor:factor-conditions} does not apply.  Its necessary condition alone
cannot prove separation either.  Taking \(a=2,b=3\) gives the basic example
\[
  F(x)=x(2x+1),\qquad G(x)=x(3x+1).
\]

\begin{question}\label[question]{q:exceptional-boundary}
Find a finite local--Galois criterion for \(\cP(F,G)\) when \(UV\) is
prime-covering but neither \(U\) nor \(V\) is intersective.  In particular,
classify pairs with four distinct linear factors and pairs with nonlinear
irreducible common factors beyond \cref{cor:degree-two}.
\end{question}

These results do not give a sufficient condition for all higher-degree
pairs with common factors.  We still need to find which further congruence
conditions, together with prime-covering of \(UV\), characterize separation.

\section*{Acknowledgements}

The author thanks Yuchen Ding for bringing this problem to his attention
and for helpful discussions and references.  The author also thanks Huixi
Li for useful comments on earlier versions of the manuscript.

\section*{Declaration on the use of generative AI}

The author used OpenAI Codex (model \texttt{gpt-5.6-sol}, accessed in
2026) to help find references, edit the text, and organize the paper.
It also helped develop and test proof ideas, including the argument for
\cref{lem:uniform-saturation}.  The author then checked all cited
references against the original sources and independently reconstructed
and checked every proof.  The author also reviewed and revised all
AI-assisted text and takes full responsibility for the paper.

\end{document}